\documentclass[11pt,reqno]{amsart}
\usepackage[utf8]{inputenc}
\usepackage{a4wide}
\usepackage{appendix}

\usepackage{graphicx}
\usepackage{enumerate}
\usepackage{hyperref}
\usepackage{cleveref}
\usepackage{amsmath,amsopn,amssymb,amsfonts,stmaryrd}
\usepackage{verbatim}
\usepackage{todonotes}
\usepackage{amsthm}
\usepackage{mathtools}
\usepackage{color}
\usepackage{enumitem}
\usepackage[framemethod=TikZ]{mdframed}
\usepackage{bbm}
\usepackage{mathrsfs}
\usepackage{booktabs}
\usepackage{caption}
\usepackage{bm}
\usepackage{tensor}

\makeatletter
\@namedef{subjclassname@2020}{%
	\textup{2020} Mathematics Subject Classification}
\makeatother

\newcommand{\IR}{{\mathbb R}}
\newcommand{\IC}{{\mathbb C}}

\newcommand{\IZ}{{\mathbb Z}}

\newcommand{\CC}{\mathcal{C}}

\renewcommand{\Im}{{\rm Im}}

\newcommand{\dist}{\mathrm{dist}}

\renewcommand{\O}{\mathrm{O}}

\newcommand{\Lie}{\mathrm{Lie}}
\newcommand{\Isom}{\mathrm{Isom}}

\newcommand{\C}{\mathbb C}

\theoremstyle{plain}
\newtheorem{thm}{Theorem}[section]
\newtheorem{cor}[thm]{Corollary}
\newtheorem{lem}[thm]{Lemma}
\newtheorem{prop}[thm]{Proposition}
\newtheorem{conj}[thm]{Conjecture}

\newtheorem{rem}[thm]{Remark}

\theoremstyle{definition}

\newtheorem{defn}[thm]{Definition}

\numberwithin{equation}{section}

\def\a{\alpha}

\def\d{\delta}

\def\k{\kappa}

\def\a{\alpha}

\def\d{\delta}

\def\k{\kappa}

\def\CC{{\mathbb C}}
\def\RR{{\mathbb R}}

\def\ZZ{{\mathbb Z}}
\def\d{{\mathrm{d}}}

\def\id{\mathrm{id}}

\def\SS{{\mathbb{S}}}
\def\S{{\mathbb{S}^{2n-1}}}

\def\i{{\mathrm{i}}}

\def\Re{{\mathrm{Re}}}
\def\Im{{\mathrm{Im}}}

\def\Diff{\mathrm{Diff}}

\newcommand{\gc}{\gamma}
\newcommand{\dgc}{\dot{\gc}}

\def\bar{\overline}

\newcommand{\comma}[0]{\, , \,}

\newcommand{\set}[2]{\left\{ \, #1  \, | \,  #2 \, \right\}}

\newcommand{\Unitary}[0]{\mathrm{U}}

\makeatletter
\newcommand{\vast}{\bBigg@{2}}
\newcommand{\Vast}{\bBigg@{5}}
\makeatother

\newcommand{\RNum}[1]{\uppercase\expandafter{\romannumeral #1\relax}}
\allowdisplaybreaks

\title[Topics in Magnetic Geometry]{Topics in Magnetic Geometry:\\ interpolation, intersections and integrability} 

\author{Lina Deschamps}
\address{Faculty of Mathematics and Computer Science,
	University of Heidelberg,
	Im Neuenheimer Field 205,
	69120 Heidelberg, Germany}
\email{ldeschamps@mathi.uni-heidelberg.de}

\author{Levin Maier}
\address{Faculty of Mathematics and Computer Science,
	University of Heidelberg,
	Im Neuenheimer Field 205,
	69120 Heidelberg, Germany}
\email{lmaier@mathi.uni-heidelberg.de}

\author{Tom Stalljohann}
\address{Faculty of Mathematics and Computer Science,
	University of Heidelberg,
	Im Neuenheimer Field 205,
	69120 Heidelberg, Germany}
\email{tstalljohann@mathi.uni-heidelberg.de}

\keywords{}

\subjclass[2020]{37J35,37J55, 53D25}

\begin{document}
	\maketitle
\renewcommand{\abstractname}{Abstract}
\begin{abstract}
This paper develops new links between contact geometry, magnetic dynamics, and symmetry in exact magnetic systems.

First, we establish an interpolation property for Killing magnetic systems on contact manifolds under an additional condition. Specifically, we show that the corresponding magnetic geodesic flow interpolates smoothly between the sub-Riemannian geodesic flow on the contact distribution and the flow of the vector field associated with a primitive of the magnetic field.

Second, we show that Hamiltonian group actions associated with the magnetomorphism group produce Poisson-commuting integrals of motion for the magnetic flow.

Finally, we obtain new structural results on totally magnetic submanifolds, showing that fixed-point sets of magnetomorphisms and intersections of totally magnetic submanifolds are again totally magnetic. The latter two results may be viewed as extensions of classical phenomena from Riemannian geometry to magnetic geometry.
\end{abstract}
	

\section{Introduction}
\label{section: Introduction}

After the pioneering work of Arnold~\cite{ar61} in the 1960s, which placed the motion of a charged particle in a magnetic field into the framework of modern dynamical systems, magnetic flows have attracted considerable attention; see, for instance, \cite{AbbMacMazzPat17,Abbondandolo2015,Assenza24,AssBenLust16,CIPP,Fat97,Man,Me09,Sor,Tai83,Tai92a,Tai92b}.\medskip

\paragraph{\textbf{Magnetic geodesic systems.}}
Let $(M,g)$ be a closed, connected Riemannian manifold and $\sigma\in\Omega^2(M)$ a closed two-form. The form $\sigma$ is called a \emph{magnetic field}, and the triple $(M,g,\sigma)$ is called a \emph{magnetic system}. It determines a skew-symmetric bundle endomorphism $Y\colon TM\to TM$, the \emph{Lorentz force}, by
\begin{equation}\label{e:Lorentz}
	g_q\left(Y_qu,v\right)=\sigma_q(u,v),\qquad \forall\, q\in M,\ \forall\,u,v\in T_qM.
\end{equation}
A smooth curve $\gamma\colon \RR\to M$ satisfying
\begin{equation}\label{e:mg}
	\nabla_{\dot\gamma}\dot\gamma= Y_{\gamma}\dot\gamma
\end{equation}
is called a \emph{magnetic geodesic} of $(M,g,\sigma)$. Here $\nabla$ denotes the Levi-Civita connection of $g$. When $\sigma=0$, this reduces to the usual geodesic equation. A central question is therefore to understand similarities and differences between standard and magnetic geodesics.

Since $Y$ is skew-symmetric, magnetic geodesics preserve the kinetic energy
\[
E(\gamma,\dot\gamma):=\tfrac12 g_\gamma(\dot\gamma,\dot\gamma),
\]
and hence have constant speed. This reflects the Hamiltonian nature of the system. Indeed, the \emph{magnetic geodesic flow} is defined by
\[
\varPhi_{g,\sigma}^t\colon TM\to TM,\quad (q,v)\mapsto \left(\gamma_{q,v}(t),\dot\gamma_{q,v}(t)\right),\quad \forall t\in\RR,
\]
where $\gamma_{q,v}$ is the magnetic geodesic with initial condition $(q,v)$. As shown in~\cite{Gin}, this flow is Hamiltonian with Hamiltonian $E$ and twisted symplectic form
\begin{equation}\label{eq:twisted_symplectic_structure}
    \omega_\sigma = \d\lambda - \pi^*_{TM}\sigma.
\end{equation}
A key difference from the geodesic flow is that magnetic geodesics with different speeds cannot be obtained from one another by reparametrization. This motivates the study of the flow on varying energy levels.\\

\paragraph{\textbf{Magnetic geometry.}}
Magnetic systems admit a geometric interpretation in which classical notions from Riemannian geometry—such as curvature, isometry groups, and totally geodesic submanifolds—have natural magnetic counterparts and interact with the dynamics.

While the notion of magnetic curvature has recently been introduced in~\cite{Assenza24}, our focus is on the relation between symmetry and invariant submanifolds. In~\cite{ABM}, Albers, Benedetti, and the second author introduced \emph{totally magnetic submanifolds} and the \emph{magnetomorphism group}. The former are submanifolds invariant under the magnetic flow; see \Cref{defn: totally magnetic submfds}. The latter consist of diffeomorphisms preserving both $g$ and $\sigma$; see \Cref{defn: magnetomorphism group}. These parallels are summarized in Table~\ref{table:1_comp_riem_magn_geom}, which also highlights the magnetic counterparts established in the present work.

\begin{table}[h]
\centering
\begin{tabular}{|p{0.43\textwidth}|p{0.43\textwidth}|}
\hline
\textbf{Riemannian geometry} & \textbf{Magnetic geometry} \\
\hline
Geodesics & Magnetic geodesics \\
\hline
Curvature & Magnetic curvature~\cite{Assenza24} \\
\hline
Totally geodesic submanifolds & Totally magnetic submanifolds~\cite{ABM} \\
\hline
Fixed-point sets of isometries are totally geodesic submanifolds & Fixed-point sets of magnetomorphisms are totally magnetic submanifolds~(\Cref{prop: totally magnetic submfds - fixed point set}) \\
\hline
Intersections of totally geodesic submanifolds are totally geodesic & Intersections of totally magnetic submanifolds are totally magnetic~(\Cref{prop: totally magnetic submfds - intersections}) \\
\hline
The Lie algebra of the isometry group encodes first integrals of the geodesic flow & The Lie algebra of the magnetomorphism group encodes first integrals of the magnetic geodesic flow~(\Cref{thm:lie_algebra_mag_group_encodes_integrals_of_motion}) \\
\hline
\end{tabular}
\caption{Parallel statements in Riemannian geometry and magnetic geometry, including the results established in this article.}
\label{table:1_comp_riem_magn_geom}
\end{table}

The study of totally magnetic submanifolds has attracted further attention in \cite{Terek_submanifold_compatibility_equations_magnetic,Marshall_Terek_Totally_Magnetic_Hypersurface} and has found applications to nonlinear PDEs, notably the magnetic two-component Hunter--Saxton system~\cite{M24}. This provides additional motivation for the present work.
\medskip

\paragraph{\textbf{Main contributions.}}
We establish the following results:
\begin{itemize}
    \item \textbf{Interpolation property.} For Killing magnetic systems on contact manifolds, we prove that, under a natural compatibility condition, the magnetic geodesic flow interpolates between the sub-Riemannian geodesic flow on $(M,\xi=\ker\alpha)$ and the flow of the vector field dual to $\alpha$; see \Cref{prop: general interpolation}.
    
    \item \textbf{Integrals of motion from symmetry.} For exact magnetic systems, we show that Hamiltonian group actions associated with the magnetomorphism group yield Poisson-commuting, functionally independent integrals of motion, with cardinality at least the dimension of an enclosed 
    torus; see \Cref{thm:lie_algebra_mag_group_encodes_integrals_of_motion} and \Cref{thm: Hamiltonian action of ExMag}.
    
    \item \textbf{Structure of totally magnetic submanifolds.} We prove that fixed-point sets of magnetomorphisms are totally magnetic submanifolds and that intersections of totally magnetic submanifolds remain totally magnetic; see \Cref{prop: totally magnetic submfds - fixed point set} and \Cref{prop: totally magnetic submfds - intersections}.
\end{itemize}\medskip

\paragraph{\textbf{Outline of the article.}}
\Cref{section: Preliminaries} reviews variational aspects of magnetic flows and Mañé’s critical values. \Cref{section: Killing magnetic systems} recalls Killing magnetic systems and establishes the interpolation property, with applications to contact geometry. \Cref{section: Illustration on Ellipsoids} provides examples on complex ellipsoids and derives explicit integrals of motion. \Cref{section: Magnetic Geometry of Hypersurfaces} develops the relation between symmetry and integrability via Hamiltonian group actions. \Cref{subsection: magnetic geometry - abstract results} contains structural results on totally magnetic submanifolds.
\medskip
 
	\noindent \textbf{Acknowledgments:}
The authors are grateful to their advisor Peter Albers for helpful discussions.
The authors acknowledge funding by the Deutsche Forschungsgemeinschaft (DFG, German Research Foundation) – 281869850 (RTG 2229), 390900948 (EXC-2181/1) and 281071066 (TRR 191).  L.D. and L.M. would like to acknowledge the excellent working
conditions and interactions at Erwin Schrödinger International
Institute for Mathematics and Physics, Vienna, during the thematic
programme \emph{``Infinite-dimensional Geometry: Theory and Applications"}
where part of this work was completed.

\section{Preliminaries}
\label{section: Preliminaries}

\subsection{Variational description of magnetic geodesic flows}\label{ss: Intermezzo magnetic systems}

In the case of an exact magnetic system \((M,g,\d\alpha)\), the magnetic
geodesic flow admits a Lagrangian description, and hence a variational
formulation. The associated magnetic Lagrangian is
\[
L \colon TM \to \mathbb{R},
\qquad
L(q,v):=\tfrac{1}{2}|v|^2-\alpha_q(v).
\]
The magnetic geodesic flow \(\Phi^t_{g,\d\alpha}\) coincides with the
Euler--Lagrange flow \(\Phi_L\) associated with the magnetic Lagrangian \(L\);
see~\cite{Gin}. More precisely, a curve \(\gamma\colon [0,T]\to M\) is a
magnetic geodesic if and only if it is a critical point of the action
functional
\[
S_L(\gamma):=\int_0^T L(\gamma(t),\dot\gamma(t))\,\d t
\]
among all smooth curves \(\delta\colon [0,T]\to M\) satisfying
\[
\delta(0)=\gamma(0),
\qquad
\delta(T)=\gamma(T).
\]

This variational principle prescribes the time interval \([0,T]\), while the
energy of \(\gamma\) is left free. On the other hand, for a fixed energy level
\(\kappa\in\mathbb{R}\), a curve \(\gamma\) is a magnetic geodesic of energy
\(\kappa\) if and only if it is a critical point of the free-time action
functional \(S_{L+\kappa}\), where both the curve and the travel time are
allowed to vary. More precisely, one considers variations among smooth curves
\(\delta\colon [0,T']\to M\) with arbitrary \(T'>0\) such that
\[
\delta(0)=\gamma(0),
\qquad
\delta(T')=\gamma(T).
\]

We conclude this subsection by recalling that a magnetic geodesic \(\gamma\) of
\((M,g,\d\alpha)\) is said to have \emph{negative action} if
\[
S_{L+\kappa}(\gamma)<0,
\]
where \(L\) is the magnetic Lagrangian of \((M,g,\d\alpha)\) and \(\kappa\) is
the energy of \(\gamma\).

\subsection{Mañé's critical values}

This variational formulation underlies the definition of \emph{Mañé's critical
values}, introduced in the seminal works~\cite{CIPP,Man}. These quantities may
be interpreted as energy levels at which significant dynamical and geometric
transitions occur in the Euler--Lagrange flow associated with the magnetic
Lagrangian \(L\). For simplicity, we restrict here to the definition of the Mañé critical value,
defined by
\begin{equation}\label{eq: strict mane value}
    c(L):=
    \inf\left\{
    \kappa\in\mathbb{R}
    \,\middle|\,
    S_{L+\kappa}(\gamma)\geq 0
    \text{ for all } T>0
    \text{ and all }
    \gamma\in C^\infty(\mathbb{R}/T\mathbb{Z},M)
    \right\}.
\end{equation}
For an overview of other Mañé critical values, we refer to~\cite{Abbo13Lect}
and the references therein. Finally, we note that the Mañé critical value can
also be defined for non-exact magnetic fields, following
\cite{CFP10,Me09}, although this generalization lies beyond the scope of the
present paper.

\subsection{Background in magnetic geometry.}
In Riemannian geometry, the study of totally geodesic submanifolds is a classical topic. In \cite{ABM}, Albers, Benedetti, and the second author initiated the analogous study in magnetic geometry by introducing the notion of so-called totally magnetic submanifolds. The reader may think of them as the analogue of totally geodesic submanifolds for magnetic geodesic flows, in the sense that they are precisely the submanifolds that are invariant under the magnetic geodesic flow, just as totally geodesic submanifolds are precisely the submanifolds invariant under the geodesic flow. 

We begin by recalling the notion of a totally magnetic submanifold.

\begin{defn}[{Totally magnetic submanifold, \cite[Def. 6.5]{ABM}}]
\label{defn: totally magnetic submfds}
Let $(M,g,\sigma)$ be a magnetic system. An (embedded) submanifold $N \subseteq M$ is called \textit{totally magnetic} if for every magnetic geodesic $\gamma : I \to M$ with $\gamma(0) \in N$ and $\dot{\gamma}(0) \in T_{\gamma(0)}N$, there exists $\varepsilon > 0$ such that $\gamma((-\varepsilon,\varepsilon)) \subseteq N$.
\end{defn}
Before proceeding, we recall the following useful criterion for determining whether a submanifold is totally magnetic.
\begin{thm}[\cite{ABM}, Thm.~1.4]
\label{t:totally_magnetic_submanifolds_characterization_ABM}
Let $(M,g,\sigma)$ be a magnetic system with Lorentz force $Y \colon TM \to TM$, and let $N \subseteq M$ be a submanifold. Denote by $i \colon N \hookrightarrow M$ the inclusion.
The following are equivalent:
\begin{enumerate}[label=(\roman*)]
    \item $N$ is a totally magnetic submanifold of $(M,g,\sigma)$.
    \item Every magnetic geodesic $\gamma \colon I \to N$ of $(N, i^*g, i^*\sigma)$ is also a magnetic geodesic of $(M,g,\sigma)$.
    \item $N$ is a totally geodesic submanifold of $(M,g)$ and $Y_p(T_pN) \subseteq T_pN$ for all $p \in N$.
\end{enumerate}
\end{thm}

Typically, totally geodesic submanifolds are classified up to isometry. For magnetic systems, the analogous notion is a structure-preserving map of magnetic systems, namely a \emph{magnetomorphism}, which is defined as follows.

\begin{defn}[{Magnetomorphism group, \cite[Def.~6.2]{ABM}}]
\label{defn: magnetomorphism group}
Given two magnetic systems $(M_1,g_1,\sigma_1)$ and $(M_2,g_2,\sigma_2)$, the two systems are called \emph{magnetomorphic} if there exists a diffeomorphism $\varPhi: M_1 \to M_2$ such that
\[
\varPhi^*g_2=g_1 \quad \text{and} \quad \varPhi^*\sigma_2=\sigma_1\,.
\]Such a map is called a \emph{magnetomorphism}.\\

\noindent For a magnetic system $(M,g,\sigma)$, the \textit{group of magnetomorphisms} is
\begin{align*}
    \mathrm{Mag}(M,g,\sigma)
    :=
    \set{\psi \in \mathrm{Diff}(M)}{\psi^*g = g \text{ and } \psi^* \sigma = \sigma}\, .
\end{align*}
\end{defn}

\subsection{Contact geometry}
Natural choices for magnetic forms are symplectic forms or contact forms, where the latter arise in contact geometry, the odd-dimensional analogue of symplectic geometry. We begin by recalling the basic notions. 

A \emph{contact manifold} is a pair \((M,\alpha)\), where
\(M\) is a closed \((2n+1)\)-dimensional manifold and
\(\alpha\in\Omega^1(M)\) is a \(1\)-form such that
\[
\alpha\wedge (\d\alpha)^n \neq 0
\]
everywhere on \(M\). There exists a unique vector field \(R\), called the
\emph{Reeb vector field}, such that
\[
\alpha(R)=1,
\qquad
\iota_R \d\alpha = 0.
\]
The contact form \(\alpha\) gives rise to the contact distribution $\xi:=\ker \alpha$,
which is a completely non-integrable hyperplane distribution. Moreover, one has
the splitting
\begin{equation*}\label{eq:splitting_of_TM_into_xi_R}
    TM=\xi\oplus\langle R\rangle_{\RR}
\end{equation*}
and the kernel of \(\d\alpha\) is precisely the line bundle generated
by \(R\), that is,
\begin{equation*}\label{eq: kernel_of_alpha}
    \ker \d\alpha= \langle R\rangle_{\RR}.
\end{equation*}

For more details on contact geometry we refer the reader to~\cite{Gg08}. 
It is natural to consider contact structures that are compatible with a Riemannian metric. This leads, in particular, to the notion of a \(K\)-contact structure, which links contact geometry with Riemannian geometry; see~\cite{Blair}.\\

A \emph{\(K\)-contact structure} on \(M\) is a pair \((g,\alpha)\), where
\(\alpha\) is a contact form and \(g\) is a Riemannian metric satisfying the
following properties:
\begin{enumerate}[label=(K\arabic*),ref=K\arabic*]
    \item\label{it:1_K_contact}
    the contact distribution \(\ker\alpha\) is $g$-orthogonal to the Reeb vector
    field \(R\);
    \item\label{it:2_K_contact}
    the Reeb vector field \(R\) is a Killing vector field for \(g\) and has
    constant norm \(r>0\).
\end{enumerate}

\section{Geometry of Killing magnetic systems}
\label{section: Killing magnetic systems}
This section focuses on Killing magnetic systems, and illustrates the strong geometric constraints on magnetic trajectories imposed by the Killing condition. We show an interpolation property of the magnetic geodesic flow on Killing magnetic systems, which is of independent interest, particularly considering applications to $K$-contact manifolds.\\

We recall the definition of a Killing magnetic system introduced in~\cite{ABM}.

\begin{defn}[{\cite[Def. 2.1.]{ABM}}]
\label{defn: Killing}
An exact magnetic system $(M,g,\d\alpha)$ on a closed manifold $M$ is called \emph{Killing} if the $g$-dual vector field $X$ of $\alpha$ is a Killing vector field, that is, the flow of $X$ acts by isometries on $(M,g)$. 
In this case, we say that $\alpha$ is Killing with respect to $g$, and that the associated magnetic Lagrangian $L$ is Killing.
\end{defn}

\subsection{The Interpolation for Killing Magnetic Systems}
\label{section: Interpolation Property for Killing Magnetic Systems}

In this subsection, we study an \emph{interpolation property}, generalizing~\cite[Cor.~1.8]{ABM}. We characterize the Killing magnetic systems whose flow interpolates between a sub-Riemannian geodesic flow and a Reeb flow. For background on sub-Riemannian geometry see~\cite{Montgomereysubriembook}.

This property can be detected by the angle between $X$, the vector field dual to the magnetic $1$-form $\alpha$, and the Reeb direction along magnetic geodesics. We show that it holds precisely when the norm of $X$ stays constant along magnetic geodesics.

\begin{thm}[Interpolation property]\label{prop: general interpolation}
Let \((M^{2n+1}, \alpha)\) be a closed contact manifold equipped with a Riemannian metric \(g\), and suppose that the magnetic system \((M, g, \mathrm{d}\alpha)\) is Killing.  If the $g$-metric dual $X$ of $\alpha$ satisfies
\begin{equation}\label{it:prop hyp}
    \d\alpha(X,\cdot)=0 \, ,
\end{equation}
then the $g_{\gamma}(X_{\gamma}, \dot{\gamma})$ is conserved along magnetic geodesics $\gc$ of $(M,g,\d\alpha)$.\\
Moreover, the magnetic geodesic flow of \( (M, g, \mathrm{d}\alpha) \) interpolates smoothly between the sub-Riemannian geodesic flow on the distribution \( \xi := \ker \alpha \) and the flow of \( X \).

More specifically:
\begin{enumerate}
    \item \label{it:prop sR}If \( \gamma \) is a magnetic geodesic with initial velocity \( \dot{\gamma}(0) \in \xi_{\gc(0)} \), then \( \gamma \) is a sub-Riemannian geodesic tangent to the distribution \( \xi \).
    
    \item \label{it:prop 'Reeb'}If \( \gamma \) is a magnetic geodesic with initial velocity \( \dot{\gamma}(0) = r \cdot X_{\gamma(0)} \) for some \( r \in \mathbb{R} \), then \( \gamma \) is a flow line of the vector field \( r \cdot X \).
\end{enumerate}
\end{thm}
\begin{rem}
    As a special case of \Cref{prop: general interpolation}, we recover the recent result~\cite[Cor. 1.8]{ABM} of the second author, which provides the first example of a higher-dimensional magnetic system having the interpolation property.
\end{rem}
\begin{rem}
Such an interpolation phenomenon is specific to magnetic dynamics: for the ordinary geodesic flow of a Riemannian metric, there is in general no analogous mechanism interpolating between a sub-Riemannian geodesic flow on a hyperplane distribution and the flow of a distinguished vector field.
\end{rem}

\begin{proof}
We begin by showing that \( g_{\gamma}(X_{\gamma}, \dot{\gamma}) \) is a conserved quantity along any magnetic geodesic \( \gamma \) of \( (M, g, \mathrm{d}\alpha) \).

By differentiating \( g_{\gamma}(X_{\gamma}, \dot{\gamma}) \) with respect to \(t\) and using that \(\gamma\) is a magnetic geodesic of \((M,g,\d\alpha)\), hence a solution of \eqref{e:mg}, we obtain
\begin{equation}\label{eq: derivative of contact angle}
    \frac{\mathrm{d}}{\mathrm{d} t} g_{\gamma}(X_{\gamma}, \dot{\gamma}) 
    = g_{\gamma}\left(\nabla_{\dot{\gamma}} X_{\gamma}, \dot{\gamma} \right) 
    + g_{\gamma}\left(X_{\gamma}, Y_{\gamma}(\dot{\gamma}) \right).
\end{equation}
Since the magnetic system \((M,g,\d\alpha)\) is Killing, we have, by \cite[S.
2.2]{ABM},
\begin{equation}\label{eq: Y=2dX}
    Y = 2\nabla X \, .
\end{equation}
Next, combining \eqref{eq: Y=2dX} with the fact that $\dot{\gamma}$ is perpendicular to $Y_\gamma(\dot{\gamma}) $ due to the skew-symmetry of the Lorentz force, we conclude that the first term on the right-hand side of \eqref{eq: derivative of contact angle} vanishes.
Using \eqref{eq: Y=2dX} again, the second term is equal to
\begin{equation}\label{eq:2nd term}
2 \cdot g_{\gamma}(X_{\gamma}, \nabla_{\dot{\gamma}} X) = \frac{\mathrm{d}}{\mathrm{d}t} \, g_{\gamma}(X_{\gamma}, X_{\gamma}) = \frac{\mathrm{d}}{\mathrm{d}t} \lvert X_{\gamma} \rvert_g^2 \, .
\end{equation}
By~\cite[Lemma 2.2]{ABM}, the assumption \eqref{it:prop hyp} on the dual vector field $X$ is equivalent to the norm \( \lvert X \rvert_g \) being constant on $M$.
Hence, the second term \eqref{eq:2nd term} vanishes as well, which concludes the proof that $g_{\gamma}(X_{\gamma}, \dot{\gamma})$ is an integral of motion.
\medskip

\textit{We now prove the interpolation property.} Let $\gamma$ be a magnetic geodesic with initial velocity $\dgc(0)\in \xi_{\gc(0)}=\ker\alpha_{\gc(0)}$, i.e., $\gamma$ is as in (\ref{it:prop sR}). Since $X$ is $g$-metric dual of $\a$ with respect to $g$, by definition, $X$ is $g$-orthogonal to the distribution $\xi:=\ker\alpha$. Since $g_{\gamma}(X_{\gamma}, \dot{\gamma})$ is an integral of motion, the velocity of $\gamma$ is horizontal, i.e., $\dgc\in \xi_{\gc}$.
As a horizontal magnetic geodesic, $\gamma$ is a minimizer of the action $S$ among all smooth curves such that $\dgc \in \xi$, where 
\[
	S(\gamma)= \int_0^T \frac{1}{2}g_{\gc(t)}(\dgc(t),\dgc(t)) \, \d t
\]
Thus, $\gamma$ is a sub-Riemannian geodesic in $(M, \xi=\ker \alpha)$.

Now let \( \gamma \) be a magnetic geodesic with initial velocity \( \dot{\gamma}(0) = r \cdot X_{\gamma(0)} \) for some \( r \in \mathbb{R} \), i.e., $\gamma$ is as in (\ref{it:prop 'Reeb'}).
The angle \( \theta(t) \in \IR/2\pi \IZ \) between \( \dot{\gamma}(t) \) and \( X_{\gamma(t)} \) is defined via the relation
\begin{equation}\label{eq: def angle veleocity and X}
    g_{\gamma(t)}(X_{\gamma(t)}, \dot{\gamma}(t)) = \cos(\theta(t)) \cdot \lvert \dot{\gamma}(t) \rvert_g \cdot \lvert X_{\gamma(t)} \rvert_g.
\end{equation}
Since both \( \lvert \dot{\gamma}(t) \rvert_g \) and \( \lvert X_{\gamma(t)} \rvert_g \) are constant in time, it follows that \( \cos(\theta(t)) \) is also conserved. From the initial condition \( \dot{\gamma}(0) = r \cdot X_{\gamma(0)} \), we have \( \cos(\theta(0)) = 1 \), hence \( \cos(\theta(t)) = 1 \) for all \( t \). This and that the norms of $\dot{\gamma}(t)$ and $X_{\gamma(t)}$ are conserved respectively imply that
\[
    \dot{\gamma}(t) = r \cdot X_{\gamma(t)} \quad \text{for all } t \in \mathbb{R},
\]
which completes the proof.
\end{proof}

\subsection{Applications of the interpolation property}

This subsection presents $K$-contact manifolds as examples of magnetic systems to which
\Cref{prop: general interpolation} applies.

\begin{cor}
Let \((g,\alpha)\) be a \(K\)-contact structure on a smooth closed manifold
\(M\). Then the magnetic geodesic flow of \((M,g,\d\alpha)\) interpolates
between the sub-Riemannian geodesic flow on \((M,\xi:=\ker\alpha)\) and the
Reeb flow of \((M,\alpha)\).
\end{cor}

\begin{proof}
By \cite[§2.3]{ABM}, specifically the discussion immediately preceding
\cite[Theorem~2.6]{ABM}, if \((M,g,\alpha)\) is a \(K\)-contact manifold, then
the assumptions of \Cref{prop: general interpolation} are satisfied. The
claimed interpolation statement therefore follows directly from
\Cref{prop: general interpolation}.
\end{proof}

\section{Illustration on Ellipsoids}
\label{section: Illustration on Ellipsoids}

The aim of this section is to illustrate \Cref{section: Killing magnetic systems} by providing an example of a magnetic system on a contact manifold that is Killing, but whose Reeb flow does not act by isometries. This section is divided into three parts. In the first part, \Cref{ss:The spherical point of view}, we approach these systems from a spherical perspective. This allows us to identify the magnetic system on the ellipsoid with a suitable magnetic system on the unit sphere in \Cref{lem:mag_Sphere_ellipsoid}. We then show in \Cref{lemm:mag_sys_Ellipsoid_are_Killing} that this system is Killing, which allows us to compute its Mañé critical value. Next, in \Cref{ss:Integrals of motion of electromagentic Lagrangians}, we use this point of view to construct integrals of motion for the corresponding electromagnetic Lagrangians; see \Cref{Prop:intergrals of motion of electromagnetic ELE on ellipsoid}. Finally, in \Cref{ss:The contact type property of the magnetic system}, we conclude with an application of independent interest, namely to the contact type problem for low-energy surfaces.\\

We begin by introducing the setting. Let $0<a_1\leq \ldots \leq a_n$ be an increasing sequence of real numbers and set $A:=\mathrm{diag}(a_1,\ldots,a_n)$. The associated complex ellipsoid is
\begin{equation}\label{eq:defi_ellipsoid}
    E(A):=\Big\{z\in\IC^n \mid \langle A^{-1}z,z\rangle=1\Big\}
      =\Big\{(z_1,\ldots,z_n)\in\IC^n \mid \tfrac{|z_1|^2}{a_1}+\ldots+\tfrac{|z_n|^2}{a_n}=1\Big\},
\end{equation}
where $\langle\cdot,\cdot\rangle$ denotes the standard Hermitian product on $\IC^n$. The metric $g=\mathrm{Re}\langle\cdot,\cdot\rangle$ is the restriction of the Euclidean metric to $E(A)$. The magnetic potential is the one-form 
\[
\alpha_z(\cdot)=\tfrac12\,\mathrm{Re}\,\langle \i z,\cdot\rangle = \tfrac12\,\mathrm{Im}\,\langle z,\cdot\rangle \, \quad \text{for} \ z\in \CC^n.
\]
In order to study this system in more detail, it is helpful to use the following alternative point of view, which will be useful for constructing integrals of motion for this magnetic system. 
\subsection{The spherical point of view.}\label{ss:The spherical point of view}
We first fix some notation. The matrix $A$ defines the weighted Hermitian product
\begin{equation}\label{eq:weighted_product}
\langle v, w \rangle_A := \langle Av, w \rangle 
\qquad \forall\, v,w \in \C^n ,
\end{equation}
on $\C^n$. In the special case $A=\mathrm{diag}(1,\ldots,1)$, the product \eqref{eq:weighted_product} coincides with the standard Hermitian product. The weighted Hermitian product induces an exact magnetic system on the unit sphere $\mathbb{S}^{2n-1}\subset\C^n$ (with respect to the standard Hermitian product $\langle\cdot,\cdot\rangle$). By restricting $\Re\langle\cdot,\cdot\rangle_A$ to $T\mathbb{S}^{2n-1}$, we obtain a Riemannian metric $g_A$ on $\mathbb{S}^{2n-1}$. Moreover, restricting the one-form
\[
(\alpha_A)_z := \tfrac12\,\Im\langle z,\cdot\rangle_A
\]
to $T\mathbb{S}^{2n-1}$ defines a magnetic potential, which is in fact a so-called tight contact form on $\mathbb{S}^{2n-1}$. We refer to \cite{Gg08} for more background on contact geometry.

It turns out that the magnetic system $(\mathbb{S}^{2n-1},g_A,\d\alpha_A)$ is exact magnetomorphic to $(E(A),g,\d\alpha)$ in the sense of \Cref{defn: exact magn group}, that is also the primitive of the magnetic field is preserved, as stated in the following: 

\begin{lem}\label{lem:mag_Sphere_ellipsoid}
The exact magnetic systems $(E(A), g, \d\alpha)$ and $(\mathbb{S}^{2n-1}, g_A, \d\alpha_A)$ are exact magnetomorphic. In fact, the map
\[
F_A : \SS^{2n-1} \longrightarrow E(A), 
\qquad 
z \mapsto \sqrt{A}\cdot z
\]
is an exact magnetomorphism. 
\end{lem}

\begin{proof}
One readily checks that $F_A$ is a diffeomorphism. A short computation shows that
\[
F_A^* g = g_A 
\qquad \text{and} \qquad 
F_A^* \alpha = \alpha_A .
\]
Hence, by \Cref{defn: exact magn group}, $F_A$ is an exact magnetomorphism. 
\end{proof}

Using this point of view allows us to prove, for instance, that the magnetic system $(E(A), g,\d\alpha)$ is Killing and to compute its Mañé critical value, which is the content of the next lemma.

\begin{lem}\label{lemm:mag_sys_Ellipsoid_are_Killing}
The magnetic system $(\mathbb{S}^{2n-1}, g_A, \d\alpha_A)$ is Killing; consequently, so is $(E(A),g,\d\alpha)$. In particular, the Mañé critical value is given by
\[
c(E(A),g,\d\alpha)
=\frac{1}{2}\|\alpha\|_\infty^2=\frac{a_n}{8}\, ,
\]
where $\|\alpha\|_\infty$ denotes the dual norm of $\alpha$ with respect to $g$.
\end{lem}

\begin{rem}
As $\pi_1(E(A))=0$, the lowest, strict, and Mañé critical values of $(E(A),g,\d\alpha)$ coincide; see~\cite{FathiMaderna2007}.
\end{rem}

\begin{proof}
The $g_A$-metric dual of $\alpha_A$ is the vector field $X_z= \tfrac{1}{2}\i z$ for all $z\in \S$. Its flow is given by
\begin{equation}\label{eq:flow_of_X}
    \varPhi_X^t(z)=e^{\frac{1}{2}\i t}z ,
\end{equation}
which is an isometry for the metric $g_A$. Hence the magnetic system $(\S,g_A,\d\alpha_A)$ is Killing. By \Cref{lem:properties_exact_magnetomorphisms}, the Killing property is preserved under exact magnetomorphisms. Since \Cref{lem:mag_Sphere_ellipsoid} shows that $(\S,g_A,\d\alpha_A)$ and $(E(A),g,\d\alpha)$ are exact magnetomorphic, the magnetic system $(E(A),g,\d\alpha)$ is Killing as well.

Since $(\S,g_A,\d\alpha_A)$ is Killing, we can conclude from \cite[Thm. 2.4]{ABM} that
\[
c(\S,g_A,\d\alpha_A)= \frac{1}{2}\|\alpha_A\|_\infty^2\, .
\]
Using the fact that the $g_A$-metric dual of $\alpha_A$ is $X_z=\tfrac12 \i z$, we compute
\[
\|\alpha_A\|_\infty^2=\frac{1}{4}  \sup_{z\in \S}(g_A)_z(\i z, \i z)= \frac{1}{4}a_n\, ,
\]
which finishes the proof, since by \Cref{lem:properties_exact_magnetomorphisms} the Mañé critical value is preserved under exact magnetomorphisms.
\end{proof}

\subsection{Integrals of motion of electromagnetic Lagrangians on $E(A)$}\label{ss:Integrals of motion of electromagentic Lagrangians}
We use the spherical interpretation of the magnetic system $(E(A), g, \d\alpha)$ again to gain further insight. For this purpose we introduce coordinates on $\S$ given by
\[
z_1=\theta_1 e^{\i \varphi_1},\; \ldots ,\; z_n=\theta_n e^{\i \varphi_n},
\]
where $\varphi_j\in \IR/2\pi \ZZ$ for each $j$ and $\sum_{j=1}^n \theta_j^2=1$. \\

Consider now a potential $V$ on $\S$ which does not depend on $\varphi_1,\ldots,\varphi_n$. This gives rise to the following electromagnetic Lagrangian
\begin{equation}\label{eq:def_int_elec_mag_langr_on_S}
L:T\S \longrightarrow \RR, 
\qquad 
(q,v)\mapsto \frac12 (g_A)_q(v,v)-(\alpha_A)_q(v)-V(q)
\end{equation}
whose Euler--Lagrange flow agrees with the Hamiltonian flow with respect to the energy
\[
E_V = E+ V ,
\]
where $E$ denotes the kinetic Hamiltonian, and the twisted symplectic form $\omega_{\d \alpha}$.

We can now state the following.

\begin{prop}\label{Prop:intergrals of motion of electromagnetic ELE on ellipsoid}
The Euler--Lagrange flow of $L$ defined in \eqref{eq:def_int_elec_mag_langr_on_S} possesses $(n+1)$ independent Poisson-commuting integrals of motion. In the case $n=2$ the system is completely integrable. 

In particular, for $V=0$ the same conclusions hold for the magnetic system $(E(A), g, \d\alpha)$.
\end{prop}

\begin{rem}
For $n=2$ and $A=\mathrm{diag}(1,1)$ we recover \cite[Cor.~A.1]{BM24}.
\end{rem}

\begin{rem}
This result is generalized below in Theorem \ref{thm:lie_algebra_mag_group_encodes_integrals_of_motion}. See also the examples in Subsection \ref{subsection:illustrations_intregrals_of_motion}.
\end{rem}

\begin{proof}
Observe that the Lagrangian $L$ in \eqref{eq:def_int_elec_mag_langr_on_S} does not depend on the variables $\varphi_1,\ldots,\varphi_n$. Hence these variables are cyclic, and the corresponding conjugate momenta
\[
c_j := \partial_{\dot{\varphi}_j} L, \qquad j=1,\ldots,n,
\]
are conserved quantities, which moreover Poisson-commute pairwise.

The energy $E_V = E+ V$ is also conserved along the Euler--Lagrange flow. This yields $n+1$ independent integrals of motion. In particular, when $n=2$ the system is completely integrable.

Using \Cref{lem:mag_Sphere_ellipsoid}, for $V=0$ we obtain the statements for $(E(A),g,\d\a)$.
\end{proof}

\subsection{The contact type property of the magnetic system $E(A)$}\label{ss:The contact type property of the magnetic system}
We close this subsection with an observation of independent interest concerning the magnetic system $(E(A),g,\d\alpha)$, which has been a guiding example in the authors' work \cite{DMS25Contact,QuadraticgrowthLinaTom}. 

We briefly recall that the question of whether a low-energy surface $\Sigma_\kappa=E^{-1}(\kappa)$ of the magnetic geodesic flow is of contact type remains open for general closed manifolds. It was explicitly posed in 2010 by Macarini--Paternain in~\cite[p.2]{MP10}. We refer to \cite{DMS25Contact} and the references therein for a detailed overview and precise definitions.

\begin{cor}
The energy surfaces $\Sigma_\kappa$ associated with the magnetic system $(E(A), g,\d\alpha)$ are not of contact type for all $\kappa\in (0, \frac{a_n}{8}]$.
\end{cor}
\begin{proof}
Since the contact-type property is preserved under exact magnetomorphisms, it suffices that we prove the statement for the exact magnetomorphic system $(\S,g_A,\d\alpha_A)$. By \Cref{lem:mag_Sphere_ellipsoid}, this system is exact magnetomorphic to $(E(A),g,\d\alpha)$.

Denote by $e_1,\dots, e_n$ the standard basis of $\C^n$. For $z\in\{e_1,\dots, e_n\}$, the orbits of the flow in \eqref{eq:flow_of_X} are geodesics of $(\S,g_A)$ and, by the discussion in front of \cite[Lemma 2.2]{ABM}, magnetic geodesics of $(\S,g_A,\d\alpha_A)$. As the $g_A$-norm of $\alpha_A$ is maximal along $\gamma(t)=e^{\frac12 \i t} e_n$, a computation following the lines of \cite[§3.3]{DMS25Contact} shows that the constant-speed reparametrization $\gamma_r(t):=\gamma(r t)$ is, for all $r \leq 1$, a periodic magnetic geodesic of $(\S,g_A,\d\alpha_A)$ with energy $\kappa=r^2 \frac{a_n}{8}$ and non-positive action.

The existence of a null-homologous periodic orbit of energy $\kappa$ with non-positive action implies that the energy surface $\Sigma_\kappa$ is not of contact type. We refer the reader to \cite[App.~A]{DMS25Contact} and the original references \cite{ConMacPat2004,McDuff1987,Sullivan1976} for this statement. The corollary follows.
\end{proof}

\section{Further aspects of magnetic geometry: integrability and magnetic symmetries}
\label{section: Magnetic Geometry of Hypersurfaces}

In this section we investigate the relation between magnetic symmetries and integrability. More precisely, in \Cref{subsection: magnetic geometry- integrals of motion} we show that Hamiltonian group actions associated with the magnetomorphism group produce Poisson-commuting integrals of motion for the magnetic geodesic flow. In \Cref{subsection:The exact magnetomorphism group} we introduce the exact magnetomorphism group to which we apply the previous mechanism. This will be illustrated in \Cref{subsection:illustrations_intregrals_of_motion} on complex ellipsoids, homogeneous spaces and spheres of revolution.

\subsection{Integrals of motion from magnetic symmetries}
\label{subsection: magnetic geometry- integrals of motion}

Let $(M,g,\d\alpha)$ be an exact magnetic system on a closed manifold. We retain the notation from Section~\ref{section: Preliminaries}; in particular, $TM$ is equipped with the twisted symplectic form $\omega_{\d\alpha}$.

If $G\subseteq \Isom(M,g)$ is a closed Lie subgroup, we consider its natural lifted action on $TM$ and identify $\mathfrak g=\Lie(G)$ with the corresponding infinitesimal generators on $M$.

Subsequently, it will be useful to be aware of the following expression for the infinitesimal generator of the lifted action.

\begin{lem}
\label{lem: infinitesimal generator}
For $X\in\mathfrak g$, the infinitesimal generator $X^\#\in\mathfrak X(TM)$ of the lifted action is given by
\[
X^\#(q,v)=(X(q),\nabla_vX)
\]
with respect to the Levi--Civita splitting $T_{(q,v)}TM\cong T_qM\oplus T_qM$.
\end{lem}

\begin{proof}
This is the standard expression for the tangent lift with respect to the connection-induced splitting of $TTM$.
\end{proof}

We now state the general mechanism producing integrals of motion from magnetic symmetries.

\begin{thm}[Integrals of motion from magnetic symmetries]
\label{thm:lie_algebra_mag_group_encodes_integrals_of_motion}
Let $(M,g,\d\alpha)$ be an exact magnetic system on a closed manifold, and let
\[
G\subseteq \mathrm{Mag}(M,g,\d\alpha)
\]
be a closed Lie subgroup whose induced action on $(TM,\omega_{\d\alpha})$ is Hamiltonian, with moment map $\mu:TM\to\mathfrak g^*$.

If $T\subseteq G$ is a torus with Lie algebra $\mathfrak t$, then for every $X\in\mathfrak t$ the function
\[
F_X:TM\longrightarrow \RR, \quad (q,v)\mapsto\langle \mu(q,v),X\rangle
\]
is an integral of motion of the magnetic geodesic flow. Moreover, for all $X,Y\in\mathfrak t$ one has
\[
\{F_X,F_Y\}=0.
\]

In particular, any basis of $\mathfrak t$ yields a family of Poisson-commuting first integrals, which are independent on an open dense subset of $TM$ whenever the corresponding infinitesimal generators are. If $\dim T<\dim M$, the kinetic energy can be included as an additional independent integral.
\end{thm}

\begin{rem}
\Cref{thm:lie_algebra_mag_group_encodes_integrals_of_motion} applies whenever the induced action on $(TM,\omega_{\d\alpha})$ is Hamiltonian. In the following subsection we show that this is the case for the exact magnetomorphism group.
\end{rem}

\begin{rem}
The Hamiltonianity assumption is natural. Indeed, in the classical case of the isometry group acting on the tangent bundle of a Riemannian manifold equipped with its canonical symplectic structure, the induced action is always Hamiltonian, and the associated moment map recovers the usual Noether integrals of motion. Thus, \Cref{thm:lie_algebra_mag_group_encodes_integrals_of_motion} should be viewed as the magnetic analogue of this standard picture.
\end{rem}
\begin{proof}
Since $G$ acts by isometries, the induced action on $TM$ preserves the kinetic energy
\[
E(q,v)=\tfrac12 g_q(v,v),
\]
whose Hamiltonian flow is the magnetic geodesic flow. Hence, by Noether's theorem, each $F_X$ Poisson-commutes with $E$.

Equivariance of the moment map yields
\[
\{F_X,F_Y\}=\langle \mu,[X,Y]\rangle \qquad \forall X,Y\in\mathfrak{t},
\]
which vanishes since $\mathfrak t$ is abelian. The independence statements follow from the fact that the Hamiltonian vector field of $F_X$ is the infinitesimal generator $X^\#$.

If $\dim(T) < \dim(M) \,$, then the Hamiltonian vector field $X_E$ of the energy is linearly independent from the infinitesimal generators of a basis of $\mathfrak{t}$ on a dense subset of $TM$. Indeed, the pushforward of $X_E(q,v)$ via $\pi_{TM}$ is $v$, while the pushforward of the infinitesimal generators only depends on $q$, see Lemma \ref{lem: infinitesimal generator}.
\end{proof}

\subsection{The exact magnetomorphism group}\label{subsection:The exact magnetomorphism group}

\Cref{thm:lie_algebra_mag_group_encodes_integrals_of_motion} reduces the construction of first integrals to the existence of a Hamiltonian action on $(TM,\omega_{\d\alpha})$. We now show that, for exact magnetic systems, the exact magnetomorphism group provides such an action canonically, which we introduce for this purpose.

\begin{defn}[Exact magnetomorphism group]
\label{defn: exact magn group}Given two exact magnetic systems $(M_1,g_1,\d\alpha_1)$ and $(M_2,g_2,\d\alpha_2)$, the two systems are called \emph{exact magnetomorphic} if there exists a diffeomorphism $\varPhi: M_1 \to M_2$ such that
\[
\varPhi^*g_2=g_1 \quad \text{and} \quad \varPhi^*\alpha_2=\alpha_1\,.
\]Such a map is called a \emph{exact magnetomorphism}.\\

\noindent For an exact magnetic system $(M,g,\d\alpha)$, the \emph{exact magnetomorphism group} is defined by
\[
\mathrm{Mag}_{\mathrm{ex}}(M,g,\alpha)
:=
\{\psi\in\Diff(M)\mid \psi^*g=g \text{ and } \psi^*\alpha=\alpha\}.
\]
\end{defn}
\medskip
First, we observe that the subgroup
\[
\mathrm{Mag}_{\mathrm{ex}}(M,g,\alpha)\subseteq \mathrm{Mag}(M,g,\d\alpha)
\]
is closed. We next identify its Lie algebra.

\begin{lem}
\label{lem: Lie algebra Mag_ex}
The Lie algebra of the group of exact magnetomorphisms is given by
\[
\Lie(\mathrm{Mag}_{\mathrm{ex}}(M,g,\alpha))
=
\{X\in\mathfrak X(M)\mid X \text{ is Killing and } \mathcal L_X\alpha=0\}.
\]
\end{lem}

\begin{rem}
    \Cref{lem: Lie algebra Mag_ex} also holds in the case where $(M,g)$ is complete.
\end{rem}

\begin{proof}
Since $M$ is closed, $(M,g)$ is complete. Thus, it is well known that the Lie algebra of $\Isom(M,g)$ consists of Killing vector fields. For a Killing vector field $X$ on $M$ with its associated flow $\phi_X^t$, one has the standard identity
\[
\mathcal L_X\alpha=0
\quad\Longleftrightarrow\quad
(\phi_X^t)^*\alpha=\alpha \qquad \forall t\in\mathbb R.
\]
This is equivalent to $\phi_X^t\in \mathrm{Mag}_{\mathrm{ex}}(M,g,\alpha)$ for all $t$, hence to $X$ lying in the Lie algebra.
\end{proof}
Before discussing the action of the group of exact magnetomorphisms, we state a result of independent interest.

\begin{lem}\label{lem:properties_exact_magnetomorphisms}
    Given two exact magnetic systems $(M_1,g_1,\d\alpha_1)$ and $(M_2,g_2,\d\alpha_2)$ that are exactly magnetomorphic, the following holds:
    \begin{enumerate}[label=(\arabic*)]
        \item\label{it:1_lemm_prop_ex_magn_group} the Mañé critical values of $(M_1,g_1,\d\alpha_1)$ and $(M_2,g_2,\d\alpha_2)$ coincide,
        \item\label{it:2_lemm_prop_ex_magn_group} the magnetic system $(M_1,g_1,\d\alpha_1)$ is Killing if and only if $(M_2,g_2,\d\alpha_2)$ is.
    \end{enumerate}
\end{lem}

\begin{rem}
    First, we note that the conclusion in \ref{it:1_lemm_prop_ex_magn_group} also holds if one considers weakly exact magnetic systems; for a precise notion, we refer to \cite{CFP10} and the references therein.
\end{rem}

\begin{rem}
    We stress that, if two (weakly) exact magnetic systems are magnetomorphic but not exactly magnetomorphic, then their Mañé critical values are not necessarily the same. So, in order to preserve this dynamical energy threshold, the right notion is that of the group of exact magnetomorphisms.
\end{rem}

\begin{proof}
    This follows immediately from the definition of the group of exact magnetomorphisms, exact magnetomorphisms in \Cref{defn: exact magn group} and from the definition of the Mañé critical value~\eqref{eq: strict mane value}.
\end{proof}
We now show that the group of exact magnetomorphisms induces a Hamiltonian action, that is:
\begin{prop}
\label{thm: Hamiltonian action of ExMag}
Let $(M,g,\d\alpha)$ be an exact magnetic system on a closed manifold and set
\[
G:=\mathrm{Mag}_{\mathrm{ex}}(M,g,\alpha), \qquad \mathfrak g:=\Lie(G).
\]
Then the natural action of $G$ on $(TM,\omega_{\d\alpha})$ is Hamiltonian, with moment map given by
\[
\mu: TM\longrightarrow \mathfrak{g}^*,\quad(q,v)\mapsto \left[ X\mapsto
g_q(v,X_q)-\alpha_q(X_q)\right].
\]
\end{prop}

\begin{rem}
    \Cref{thm: Hamiltonian action of ExMag} recovers \cite[Cor.~1.5]{ABM} for the sphere endowed with the round metric and the standard contact form. In that case, the Hamiltonian action is described on the full magnetomorphism group. Since the magnetomorphisms of the sphere are restrictions of unitary transformations, they preserve the primitive as well, and therefore belong to the exact magnetomorphism group.
\end{rem}

\begin{proof}
The action is symplectic since each element of $G$ preserves both the metric and $\alpha$. For $X\in\mathfrak g$, using $\mathcal L_X\alpha=0$ and Cartan's formula, one computes 
\[
- \iota_{X^{\#}} (\pi_{TM}^* \d \alpha) = - \pi_{TM}^* (\iota_X  \d \alpha) = \pi_{TM}^*(\d \iota_X \alpha) = \d (\alpha(X) \circ \pi_{TM}) ,
\]
which implies
\[
\iota_{X^\#}\omega_{\d\alpha}
=
-\d\big(g(v,X)-\alpha(X)\big),
\]
so the stated function is a Hamiltonian for $X^\#$. Equivariance follows from the invariance of $g$ and $\alpha$ under $G$, using also that the adjoint action of $\psi \in G$ is $\mathrm{Ad}_\psi(X) = \d \psi \cdot X_{\psi^{-1}} \,$.
\end{proof}
In the proof of \Cref{thm: Hamiltonian action of ExMag} we have to use the isometry property since we pull back the canonical Liouville form to the tangent bundle via the metric. Instead, we can drop the metric and consider a Hamiltonian action on the cotangent bundle. This observation is of independent interest, since Hamiltonian group actions on the cotangent bundle have been studied extensively. Let us summarize this: 

\begin{cor}
    \label{thm: Hamiltonian action cotangent bundle}
Let $G\subseteq \mathrm{Diff}(M)$ be a finite-dimensional Lie subgroup with Lie algebra $\mathfrak{g}$, acting smoothly on $M$, and suppose that
\[
\psi^*\alpha = \alpha \qquad \forall \psi\in G.
\]
Endow the cotangent bundle $T^*M$ with the twisted symplectic form
\[
\d \lambda_{\mathrm{can}} - (\pi_{T^*M})^* \d \alpha \,\, ,
\]
where $\lambda_{\mathrm{can}}$ denotes the Liouville $1$-form on $T^*M$ and $\alpha\in \Omega^1(M)$.\\
\noindent Then the natural action of $G$ on the cotangent bundle $T^*M$,
\[
G \curvearrowright T^*M \comma \quad \psi \cdot (q,p) := (\psi(q), \, p \circ (d\psi_q)^{-1}),
\]
is Hamiltonian. A moment map is given by
\[
\mu: T^*M\longrightarrow \mathfrak{g}^*,\quad(q,p)\mapsto \left[X\mapsto p(X_q) - \alpha_q (X_q)\right].
\]
\end{cor}

\begin{rem}
  Since the natural action of $G$ on $M$ is assumed to be smooth, the one-parameter subgroups of $G$ induce complete flows on $M$, and we can identify the Lie algebra $\mathfrak{g}$ of $G$ with the vector fields on $M$ generating these flows.   
\end{rem}

\begin{proof}
The proof is completely analogous to the one of \Cref{thm: Hamiltonian action of ExMag}.
\end{proof}

\subsection{Illustrations}\label{subsection:illustrations_intregrals_of_motion}
We now discuss several examples illustrating how \Cref{thm: Hamiltonian action of ExMag} and ~\Cref{thm:lie_algebra_mag_group_encodes_integrals_of_motion} produce explicit first integrals from magnetic symmetries.
\paragraph{\textbf{Ellipsoids.}}
For $n\geq 2$, consider the magnetic system \[(\S,g_A,\d\alpha_A)\] from \Cref{section: Illustration on Ellipsoids}. First observe that the maps
\begin{equation}
\label{eq: example Hamiltonian action on ellipsoid}
\mathrm{diag}(e^{\i\varphi_1},\ldots,e^{\i\varphi_n})|_{\S},
\qquad \varphi_1,\ldots,\varphi_n\in\IR/2\pi\IZ,
\end{equation}
are $g_A$-isometries preserving $\alpha_A$, and hence by their diagonal form define a torus $T$ in
\[
T\subseteq\mathrm{Mag}_{\mathrm{ex}}(\S,g_A,\alpha_A).
\]
The Lie algebra of the torus $T$ is generated by the vector fields
\[
(X_1)_z=(\i z_1,0,\ldots,0) \quad  , \,\,\ldots \,\, ,\quad (X_n)_z=(0,\ldots,0,\i z_n),
\]
which are pointwise linearly independent on the dense open subset
\[
\{z\in \S \mid z_j\neq 0 \text{ for all } j\}.
\]
Consequently, the lifted vector fields $X_1^\#,\ldots,X_n^\#$ are linearly independent on a dense open subset of $T\S$. Theorems~\ref{thm: Hamiltonian action of ExMag} and~\ref{thm:lie_algebra_mag_group_encodes_integrals_of_motion} therefore yield $n+1$ independent integrals of motion. This provides an alternative point of view on Proposition~\ref{Prop:intergrals of motion of electromagnetic ELE on ellipsoid} in the case of vanishing potential. In fact, the integrals of motion $F_{X_1}, \ldots , F_{X_n}$ provided by \Cref{thm:lie_algebra_mag_group_encodes_integrals_of_motion} are precisely the conjugate momenta $c_1,\ldots , c_n$ appearing in the proof of \Cref{Prop:intergrals of motion of electromagnetic ELE on ellipsoid}.\\

\paragraph{\textbf{Homogeneous spaces.}}
Let $M=G/H$ be a compact homogeneous space endowed with a $G$-invariant Riemannian metric $g$ and a $G$-invariant $1$-form $\alpha$. Then the natural $G$-action on $M$ preserves both $g$ and $\alpha$, and hence defines a subgroup
\[
G\subseteq \mathrm{Mag}_{\mathrm{ex}}(M,g,\alpha).
\]
By Proposition~\ref{thm: Hamiltonian action of ExMag}, the lifted action of $G$ on $(TM,\omega_{\d\alpha})$ is Hamiltonian with moment map
\[
\mu(q,v)(X)=g_q(v,X_q)-\alpha_q(X_q).
\]
Let $T\subseteq G$ be a torus with Lie algebra $\mathfrak t$. Then for every $X\in\mathfrak t$ the functions
\[
F_X(q,v)=g_q(v,X_q)-\alpha_q(X_q)
\]
are integrals of motion which Poisson-commute. If the corresponding infinitesimal generators are generically linearly independent, Theorem~\ref{thm:lie_algebra_mag_group_encodes_integrals_of_motion} yields $\dim T$ independent first integrals, and together with the energy this gives a symmetry-based integrability mechanism on $G/H$.\\

\paragraph{\textbf{Surfaces of revolution.}}
Let $M$ be a surface of revolution with metric
\[
g=\d r^2+f(r)^2\,\d\theta^2
\]
and consider a magnetic system given by a rotationally symmetric primitive
\[
\alpha=a(r)\,\d\theta.
\]
Then the vector field $Y=\partial_\theta$ generates an $S^1$-action preserving both $g$ and $\alpha$, hence
\[
S^1\subseteq \mathrm{Mag}_{\mathrm{ex}}(M,g,\alpha).
\]
By Proposition~\ref{thm: Hamiltonian action of ExMag}, the corresponding moment map yields the conserved quantity
\[
F(q,v)=g_q(v,Y_q)-\alpha_q(Y_q).
\]
In coordinates this takes the explicit form
\[
F(r,\theta,\dot r,\dot\theta)=f(r)^2\dot\theta-a(r),
\]
which is constant along magnetic geodesics. Together with the energy
\[
E=\tfrac12\big(\dot r^2+f(r)^2\dot\theta^2\big),
\]
this provides two Poisson-commuting integrals, and hence complete integrability on regular energy levels.

\section{Totally magnetic submanifolds as fixed-point sets of magnetomorphisms}
\label{subsection: magnetic geometry - abstract results}

In this section, we relate totally magnetic submanifolds to the magnetomorphism group. We first prove two abstract results and then apply them to strictly convex hypersurfaces in Euclidean space. This leads to a description of the magnetomorphism groups of complex ellipsoids and to a family of totally magnetic submanifolds.

\subsection{Abstract theorems.}

Our first result is the magnetic analogue of the classical fact in Riemannian geometry that the fixed-point set of a collection of isometries is a totally geodesic submanifold (see \cite[Prop.~5.6.5]{Pet}).

\begin{prop}
\label{prop: totally magnetic submfds - fixed point set}
Let $\mathcal{M} \subseteq \mathrm{Mag}(M,g,\sigma)$ be a subset of magnetomorphisms. Then each connected component of the fixed-point set
\[
\mathrm{Fix}(\mathcal{M}) = \set{p \in M}{F(p) = p \quad \forall \, F \in \mathcal{M}}
\]
is a totally magnetic submanifold of $(M,g,\sigma)$.
\end{prop}

To prove the proposition, we recall the definition of the Zariski tangent space. Let $\mathcal{M} \subseteq \mathrm{Isom}(M,g)$ be a subset of isometries. For $p \in \mathrm{Fix}(\mathcal{M})$, the Zariski tangent space at $p$ is defined by
\begin{equation}
\label{eq: totally magnetic submfds - Zariski space}
\mathrm{Z}_p := \set{v \in T_p M}{ \d_p F \cdot v = v \quad \forall \, F \in \mathcal{M}} \subseteq T_p M \, .
\end{equation}
By \cite[Prop.~5.6.5]{Pet}, for $\varepsilon > 0$ sufficiently small, the map
\[
\exp_p^M : B_\varepsilon(0) \cap \mathrm{Z}_p \longrightarrow B_\varepsilon(p) \cap \mathrm{Fix}(\mathcal{M})
\]
is a well-defined slice chart. In particular, the connected components of $\mathrm{Fix}(\mathcal{M})$ are totally geodesic submanifolds with tangent space $\mathrm{Z}_p$ at $p$.

\begin{proof}[Proof of \Cref{prop: totally magnetic submfds - fixed point set}]
Since $\mathcal{M} \subseteq \mathrm{Mag}(M,g,\sigma)$ consists of isometries, the above discussion implies that the connected components of $\mathrm{Fix}(\mathcal{M})$ are totally geodesic submanifolds of $(M,g)$. By \Cref{t:totally_magnetic_submanifolds_characterization_ABM}, it remains to show that for every $p \in \mathrm{Fix}(\mathcal{M})$ the Lorentz force
\[
Y_p : T_p M \longrightarrow T_p M
\]
preserves $\mathrm{Z}_p$.

Let $v \in \mathrm{Z}_p$ and $F \in \mathcal{M}$ be arbitrary. Since $F$ is a magnetomorphism, its differential commutes with the Lorentz force, and hence
\[
\d_p F \cdot Y_p(v) = Y_p(\d_p F \cdot v) = Y_p(v),
\]
where the second equality follows from $v \in \mathrm{Z}_p$ and \eqref{eq: totally magnetic submfds - Zariski space}. 

Since $F$ was arbitrary, we conclude that $Y_p(v) \in \mathrm{Z}_p \,$.
\end{proof}
The next result shows that the class of totally magnetic submanifolds is stable under intersections.
\begin{prop}
\label{prop: totally magnetic submfds - intersections}
Let $N_1, N_2 \subseteq M$ be totally magnetic submanifolds of $(M,g,\sigma)$. Then each connected component of $$N = N_1 \cap N_2$$ is a totally magnetic submanifold of $(M,g,\sigma)$ with tangent space $T_p N_1 \cap T_p N_2$ at $p \in N$.
\end{prop}

\begin{proof}[Proof of \Cref{prop: totally magnetic submfds - intersections}]
Let $p \in N = N_1 \cap N_2$. Since $N_1$ and $N_2$ are totally geodesic submanifolds, the exponential map
\[
\exp^M_p : B_\varepsilon(0) \longrightarrow B_\varepsilon(p)
\]
restricts to a bijection from $B_\varepsilon(0) \cap T_p N_1 \cap T_p N_2$ onto $B_\varepsilon(p) \cap N$. Hence the connected components of $N$ are submanifolds with tangent space $T_p N_1 \cap T_p N_2$ at $p$.

Moreover, every magnetic geodesic of $(M,g,\sigma)$ starting at $p$ with initial velocity in $T_p N_1 \cap T_p N_2$ remains in both $N_1$ and $N_2$ for small time, and thus in the connected component of $N$ containing $p$.
\end{proof}

\subsection{Illustration}
\label{subsection: magnetic geometry- illustration}

We now apply the above results to strictly convex hypersurfaces in Euclidean space. We first determine their magnetomorphism groups and then use this description to construct totally magnetic submanifolds of complex ellipsoids.

\subsubsection{Classification of the magnetomorphisms on convex hypersurfaces}
\label{subsubsection: magnetomorphism group convex hypersurfaces}
This subsection is devoted to the characterization of the magnetomorphism group of a strictly convex hypersurface $H \subseteq \RR^{2n} , \, n \geq 2 \,$. We first describe the setting in detail. As in the case of ellipsoids in \Cref{section: Illustration on Ellipsoids}, we equip $H$ with the metric
\[
g=\mathrm{Re}\langle\cdot,\cdot\rangle|_{TH},
\]
that is, the restriction of the Euclidean metric to $H$. Moreover, restricting the $1$-form
\[
\alpha_z := \tfrac12\,\Im\langle z,\cdot\rangle,\quad z\in \RR^{2n},
\]
to $TH$ defines a magnetic potential on $H$ whose exterior derivative is the restriction of the canonical symplectic form on $\IR^{2n}$ to $H$.

In order to state the main result of this subsection, we fix the following notation. For $A \in \CC^{n \times n}$ and $b \in \C^n$, let
\begin{equation}\label{eq:affine_trafo}
    F_{A,b} : \C^n \longrightarrow \C^n, \qquad F_{A,b}(z) = Az + b
\end{equation}
denote the corresponding affine map. We now state the result on the magnetomorphism groups of strictly convex hypersurfaces:

\begin{prop}
\label{prop: magnetomorphism group}
The magnetomorphism group of the system $(H,g,\d\alpha)$ is the group of affine transformations of the form \eqref{eq:affine_trafo} whose linear part is unitary and that preserve $H$, that is,
\begin{align*}
\mathrm{Mag}(H , \, g , \, \d\alpha)=\set{F_{U,b}}{ U \in \mathrm{U}(n),\, b \in \IR^{2n},\, F_{U,b}(H) = H }\, .
\end{align*}
If the maximally inscribed ball of the bounded domain enclosed by $H$ is centered at the origin, then the magnetomorphism group of the system $(H,g,\d\alpha)$ is given precisely by the set of unitary transformations preserving $H$, that is,
\[
\mathrm{Mag}(H , \, g , \, \d\alpha)=\set{U}{ U \in \mathrm{U}(n),\, U(H) = H }\, .
\]
\end{prop}
The proof of \Cref{prop: magnetomorphism group} is based on the following: 
\begin{lem}
\label{lem:affine_rigidity_and_translation}
Let $H \subseteq \IR^{2n}$ be a strictly convex hypersurface, and let $A \in \O(2n)$, $b \in \IR^{2n}$.

\begin{enumerate}
\item\label{it:1 affine_rigidity_and_translation}If $F_{A,b}|_H = \id_H$, then $A = \mathbbm{1}$ and $b = 0$.

\item\label{it:2 affine_rigidity_and_translation} Assume $F_{A,b}(H) = H$. Let $x_0$ be the center of the maximally inscribed ball of the bounded domain enclosed by $H$. Then
\[
b = (\mathbbm{1}-A)x_0.
\]
In particular, $b=0$ if $x_0=0$.
\end{enumerate}
\end{lem}

\begin{proof}
(\ref{it:1 affine_rigidity_and_translation}) Since $H$ is a closed hypersurface, it is not contained in any proper affine subspace of $\IR^{2n}$. From $F_{A,b}|_H=\id_H$ we obtain
\[
|Ax+b|^2 = |x|^2 = |Ax|^2 \qquad \forall x\in H,
\]
hence
\[
2\langle Ax,b\rangle_{\IR} + |b|^2 = 0 \qquad \forall x\in H.
\]
Thus $AH$ is contained in the affine hyperplane
\[
\set{y \in \IR^{2n}}{2\langle y,b\rangle_{\IR} + |b|^2 = 0}.
\]
If $b\neq 0$, this is a proper affine subspace, which is impossible. Hence $b=0$. Then $Ax=x$ for all $x\in H$, so $H \subseteq \mathrm{Eig}(1,A)$. Since $H$ is not contained in a proper subspace, $\mathrm{Eig}(1,A)=\IR^{2n}$ and $A=\mathbbm{1}$.

(\ref{it:2 affine_rigidity_and_translation}) Let $F:=F_{A,b}$ and let $B$ be the bounded domain enclosed by $H$. Then $F(B)=B$. Since $F$ is an isometry,
\[
\dist(F(x_0),H)=\dist(x_0,H)=\sup_{x\in B}\dist(x,H).
\]
By uniqueness of the maximally inscribed ball, we obtain $F(x_0)=x_0$, hence $Ax_0+b=x_0$, which yields $b=(\mathbbm{1}-A)x_0$.
\end{proof}
We close this section by proving \Cref{prop: magnetomorphism group}.
\begin{proof}[Proof of \Cref{prop: magnetomorphism group}]
Recall that
\begin{equation}
\label{eq: magnetomorphism group - 2 out of 3}
\mathrm{U}(n) = \mathrm{O}(2n) \cap \mathrm{Sp}(2n).
\end{equation}
We first show the inclusion
\[
\set{F_{U,b}}{ U \in \mathrm{U}(n),\, b \in \IR^{2n},\, F_{U,b}(H) = H }
\subseteq \mathrm{Mag}(H,g,\d\alpha).
\]
Let $U \in \mathrm{U}(n)$ and $b \in \IR^{2n}$ satisfy $F_{U,b}(H)=H$. Since $U \in \mathrm{U}(n) \subseteq \mathrm{O}(2n)$, the affine map $F_{U,b}$ is a Euclidean isometry, and hence its restriction to $H$ preserves the metric $g$. Moreover, by \eqref{eq: magnetomorphism group - 2 out of 3}, $U$ is symplectic, so
\[
F_{U,b}^*\d\alpha=\d\alpha.
\]
Therefore $F_{U,b}|_H$ is a magnetomorphism of $(H,g,\d\alpha)$.

It remains to prove the converse inclusion. Let $\psi \in \mathrm{Mag}(H,g,\d\alpha)$. Then $\psi$ is in particular a Riemannian isometry of $(H,g)$. Since $H$ is strictly convex, every Riemannian isometry of $H$ extends to a Euclidean congruence; see \cite[Thm.~V]{Sacksteder}. Hence there exist $A \in \O(2n)$ and $b \in \IR^{2n}$ such that
\[
\psi = F_{A,b}|_H.
\]
In particular, $F_{A,b}(H)=H$.

We claim that $A \in \mathrm{U}(n)$. By \eqref{eq: magnetomorphism group - 2 out of 3}, it suffices to show that $A$ is symplectic. Let $v_1,v_2 \in \IR^{2n}$ be arbitrary. Since $H$ is strictly convex, its Gauss map is bijective. Hence there exists $p \in H$ such that $v_1,v_2 \in T_pH$. As $\psi$ preserves $\d\alpha$, we obtain
\[
\d\alpha(v_1,v_2)=\d\alpha(\d\psi_pv_1,\d\psi_pv_2)=\d\alpha(Av_1,Av_2).
\]
Thus $A \in \mathrm{Sp}(2n)$, and therefore $A \in \mathrm{U}(n)$.

This proves
\[
\mathrm{Mag}(H,g,\d\alpha)=\set{F_{U,b}}{ U \in \mathrm{U}(n),\, b \in \IR^{2n},\, F_{U,b}(H) = H }.
\]

Now assume that the center $x_0$ of the maximally inscribed ball is equal to $0$. Hence, by \Cref{lem:affine_rigidity_and_translation}(2), every affine map $F_{U,b}$ with $U \in \mathrm{U}(n)$ and $F_{U,b}(H)=H$ satisfies
\[
b=(\mathbbm{1}-U)x_0=0.
\]
Therefore
\[
\mathrm{Mag}(H,g,\d\alpha)=\set{U}{ U \in \mathrm{U}(n),\, U(H) = H }.
\]
This proves the proposition.
\end{proof}

\subsubsection{Ellipsoids}
\label{subsubsection: ellipsoids}

We now apply \Cref{prop: magnetomorphism group} and \Cref{prop: totally magnetic submfds - fixed point set} to the complex ellipsoids introduced at the beginning of \Cref{section: Illustration on Ellipsoids}. For this, we refer to that section for the definition of the ellipsoid $E(A)$, for $A=\mathrm{diag}(a_1,\ldots,a_n)$ with fixed sequence $0<a_1\leq \cdots \leq a_n$.

To describe the magnetomorphism group of the magnetic system $(E(A), g,\d\alpha)$, it is convenient to group together the distinct eigenvalues of $A$. Thus suppose that the distinct values among $a_1,\ldots,a_n$ are $\lambda_1<\cdots<\lambda_k$ with multiplicities $\ell_1,\ldots,\ell_k$, so that $\sum_{j=1}^k \ell_j=n$. 
Having fixed this notation, we can determine the magnetomorphism group of $E(A)$.

\begin{cor}
\label{cor: magnetomorphism group - Ellipsoids}
Let $n\geq 2$. Then the magnetomorphism group of the system $(E(A),g,\d\alpha)$ is given by
\[
\mathrm{Mag}(E(A),g,\d\alpha)
=
\mathrm{U}(\ell_1)\times\cdots\times \mathrm{U}(\ell_k).
\]
\end{cor}
\begin{rem}
\label{rem: magnetomorphisms sphere}
We recover \cite[Cor.~1.5]{ABM} from the special case $k=1$ in \Cref{cor: magnetomorphism group - Ellipsoids}, that is, when $A$ is a multiple of the identity and $E(A)$ is a sphere.
\end{rem}
\begin{proof}
Since $E(A)$ is centered at the origin, \Cref{prop: magnetomorphism group} implies
\[
\mathrm{Mag}(E(A),g,\d\alpha)
=
\set{U}{U \in \mathrm{U}(n),\, U(E(A)) = E(A)}.
\]
It therefore remains to characterize those unitary maps preserving $E(A)$.

We claim that for $U \in \mathrm{U}(n)$ the condition $U(E(A)) = E(A)$ holds if and only if there exist $U_j \in \mathrm{U}(\ell_j)$, $j=1,\ldots,k$, such that
\[
U=\mathrm{diag}(U_1,\ldots,U_k).
\]
This follows from the fact that an orthogonal matrix which leaves $E(A)$ invariant must preserve the eigenspace decomposition of $A$.
\end{proof}

We next apply \Cref{prop: totally magnetic submfds - fixed point set} to obtain totally magnetic submanifolds of $E(A)$.

\begin{cor}
\label{cor: totally magnetic submfds - Ellipsoids}
Let $V_1 \subseteq \IC^{\ell_1}, \ldots, V_k \subseteq \IC^{\ell_k}$ be complex linear subspaces. Then
\begin{equation}
\label{eq: totally magnetic submfds - Ellipsoids}
E(A)\cap (V_1\oplus\cdots\oplus V_k)
\end{equation}
is a totally magnetic submanifold of $E(A)$.
\end{cor}
\begin{proof}
For every subset $J \subseteq \{1,\ldots,n\}$, let
\[
N_J := \set{(z_1,\ldots,z_n) \in E(A)}{z_j = 0 \quad \forall\, j \in J}.
\]
Then $N_J$ is the fixed-point set of the collection of restrictions to $E(A)$ of the diagonal matrices
\[
D_j:=\mathrm{diag}(1,\ldots,1,-1,1,\ldots,1), \qquad j\in J,
\]
and each $D_j|_{E(A)}$ is a magnetomorphism of $(E(A),g,\d\alpha)$ by \Cref{cor: magnetomorphism group - Ellipsoids}. Hence, \Cref{prop: totally magnetic submfds - fixed point set} shows that $N_J$ is a totally magnetic submanifold of $E(A)$.

Now let
\[
U=\mathrm{diag}(U_1,\ldots,U_k),
\qquad
U_j\in\mathrm{U}(\ell_j).
\]
Again by \Cref{cor: magnetomorphism group - Ellipsoids}, the restriction of $U$ to $E(A)$ is a magnetomorphism. Therefore $U(N_J)$ is a totally magnetic submanifold of $E(A)$. Since every subset of the form \eqref{eq: totally magnetic submfds - Ellipsoids} is the image of some $N_J$ under such a block-diagonal unitary map, the claim follows.
\end{proof}

In the case of the sphere, it was proved in \cite[Cor.~1.5]{ABM} that every closed, connected totally magnetic submanifold of positive dimension is the intersection of the sphere with a complex linear subspace. This motivates the following conjecture.

\begin{conj}
\label{conj: totally magnetic submfds - Ellipsoids}
Every closed, connected totally magnetic submanifold of $E(A)$ of positive dimension is of the form \eqref{eq: totally magnetic submfds - Ellipsoids}.
\end{conj}

\section{Outlook}
\label{section: Outlook}

The results of this paper suggest several directions for further research. A particularly promising one arises from the ambient-space formulation of the magnetic geodesic equation for hypersurfaces.

More precisely, let $(M,g,\sigma)$ be a magnetic system with Levi-Civita connection $\nabla$ and Lorentz force $Y:TM\to TM$, and let $\Sigma=f^{-1}(c)$ be a regular level set of a smooth function $f\in C^\infty(M)$. Denote by $\nabla f$ the gradient of $f$ with respect to $g$. Then a curve $\gamma:(a,b)\to \Sigma$ is a magnetic geodesic of the induced magnetic system on $\Sigma$ if and only if it satisfies the ambient equation
\[
\nabla_{\dot \gamma} \dot \gamma
=
Y_\gamma(\dot \gamma)
+
\frac{g_\gamma\big(\dot \gamma,\, Y_\gamma(\nabla f(\gamma))-\nabla_{\dot \gamma}\nabla f\big)}{|\nabla f(\gamma)|^2}\,\nabla f(\gamma).
\]
This formula fits into the framework of \cite{Terek_submanifold_compatibility_equations_magnetic} and makes it possible to study magnetic geodesics on hypersurfaces directly in terms of ambient data.

In the case of complex ellipsoids, this leads to a concrete second-order ODE in the ambient space $\IC^n$. Ambient formulations of the magnetic geodesic equation have already proved useful in related situations. In \cite{ABM}, they led to several structural insights into the dynamics, and these ideas were subsequently used by the second author in \cite{M24} in an application to a nonlinear PDE.

It would therefore be interesting to investigate whether, in the setting of \Cref{subsection: magnetic geometry- illustration}, the ambient ODE for magnetic geodesics on ellipsoids leads to comparable new insights. This may include a better analytical understanding of the dynamics, effective numerical approaches to the magnetic flow, or possible links with nonlinear PDEs, for instance of elliptic version of a Hunter--Saxton type equation.

More broadly, the examples studied in this paper indicate that magnetic geometry on hypersurfaces may provide a useful bridge between dynamical, geometric, and analytic questions. Exploring this interaction further seems to be a natural direction for future work.

\newpage
\bibliographystyle{abbrv}
	\bibliography{ref}
\end{document}